\documentclass[11pt, reqno]{amsproc}

\usepackage{geometry}
\usepackage{color}
\usepackage{verbatim}
\usepackage[utf8]{inputenc}
\usepackage{t1enc}
\usepackage[english]{babel}

\usepackage{amsmath,amssymb,amsthm}
\usepackage{mathrsfs,esint}
\usepackage{graphicx,wrapfig}
\usepackage[format=hang]{caption}

\newcommand*{\R}{{\mathbb R}}

\newcommand*{\N}{{\mathbb N}}

\newcommand*{\eps}{\varepsilon}

\newcommand*{\Om}{\Omega}

\newcommand{\uB}{\overline{u}}
\newcommand{\vB}{\overline{v}}

\newcommand{\xB}{\overline{x}}
\newcommand{\yB}{\overline{y}}
\newcommand{\zB}{\overline{z}}

\providecommand*{\vint}[1]{\mathchoice
          {\mathop{\vrule width 5pt height 3 pt depth -2.5pt
                  \kern -9pt \kern 1pt\intop}\nolimits_{\kern -5pt{#1}}}
          {\mathop{\vrule width 5pt height 3 pt depth -2.6pt
                  \kern -6pt \intop}\nolimits_{\kern -3pt{#1}}}
          {\mathop{\vrule width 5pt height 3 pt depth -2.6pt
                  \kern -6pt \intop}\nolimits_{\kern -3pt{#1}}}
          {\mathop{\vrule width 5pt height 3 pt depth -2.6pt
                  \kern -6pt \intop}\nolimits_{\kern -3pt{#1}}}}
\newcommand*{\jint}{\fint}
\DeclareMathOperator{\Lip}{Lip}

\DeclareMathOperator{\rad}{rad}

\numberwithin{equation}{section}
\theoremstyle{plain}
\newtheorem{thm}[equation]{Theorem}

\newtheorem{prop}[equation]{Proposition}

\newtheorem{lem}[equation]{Lemma}

\theoremstyle{definition}

\newtheorem{defn}[equation]{Definition}

\newtheorem{remark}[equation]{Remark}

\begin{document}

\title[Approximation of harmonic functions via graphs]
{Approximation of harmonic functions on metric measure spaces of controlled geometry via discrete graphs} 
 \author{Almaz Butaev} 
\address{Department of Mathematics and Statistics, University of the Fraser Valley, 33844 King Road, Abbotsford, BC  V2S 7M7, Canada.}
\email{almaz.butaev@ufv.ca}
\author{Liangbing Luo} 
\address{Department of Mathematics, 48 University Avenue, Jeffery Hall, Queens University, Kingston, ON K7L 3N6, Canada.}
\email{liangbingluo94@gmail.com}
\author{Nageswari Shanmugalingam} 
\address{Department of Mathematical Sciences, P.O.~Box 210025, University of Cincinnati, Cincinnati, OH~45221-0025, U.S.A.}
\email{shanmun@uc.edu}
\thanks{
 N.S.'s work is partially supported by National Science Foundation (US) grant DMS~\#234874. A.B. acknowledges the support of the Natural Sciences and Engineering Research Council of Canada (NSERC), [funding reference number RGPIN-2025-05594]}

\begin{abstract}
Given a complete doubling metric measure space $X$ that supports a $2$-Poincaré inequality, we approximate harmonic functions on a 
bounded 
domain with a prescribed Newton-Sobolev boundary data. Our approach is based on the approximation of the underlying space $X$ by a family of graphs (see \cite{GL}).
This approximated harmonic function is realized as the weak limit of a sequence of functions obtained from the graph minimizers. We prove that such a function is a minimizer with respect to a nonlinear energy form on $N^{1,2}_0(\Om)$, which is
in turn,
majorized by the upper gradient energy on $N^{1,2}(X)$. This energy form on $N^{1,2}_0(\Om)$ is obtained as a 
$\Gamma$-limit of a sequence of induced energy forms 
projected from the discrete energy form on the approximating graphs. 
\end{abstract}

\maketitle

\noindent
    {\small \emph{Key words and phrases}: 
    Harmonic functions, graph approximation, $\Gamma$-convergence, energy form, energy minimizers, Dirichlet boundary value problem,
    discrete approximations.
}

\medskip

\noindent
    {\small Mathematics Subject Classification (2020):
Primary: 31E05
Secondary: 46E36, 49Q20, 65Z99, 
30L15.
}

\section{Introduction}
In the seminal work~\cite{Chee}, it was shown that when a doubling measure $\mu$ on a complete metric space $(X,d)$ supports a $p$-Poincar\'e
inequality with respect to the intrinsic upper gradient structure on $X$, then there is a linear differential structure $D$ on $(X,d,\mu)$ such that
every (locally) Lipschitz function $f$ on $X$ is $\mu$-almost everywhere differentiable. Moreover, there is a measurable inner product structure
on the differential such that whenever $f$ is a Lipschitz function on $X$, we have $\|Df\|\approx\Lip f$ $\mu$-a.e.~in $X$. This comparison
is also known to extend to the wider class of Newton-Sobolev functions, that is, the class $N^{1,p}(X)$, with $D:N^{1,p}(X)\to L^p(X,\R^N)$
for some fixed $N$ that depends solely on the data related to the doubling property and the Poincar\'e inequality, and
$\|Df\|\approx g_f$ $\mu$-a.e.~in $X$ when $g_f$ is the minimal $p$-weak upper gradient of $f$, see~\cite{Chee, HKSTbook}.

However, the proof of existence of the linear differential structure in~\cite{Chee} is an existence proof, and in addition, the corresponding 
chart decomposition of $X$ gives only measurable atlas of $X$ without an explicit construction of such an atlas. The coordinate chart, while
is Lipschitz continuous (locally), is only shown to exist. The first improvement on this was in~\cite{Keith}, where it was shown that 
we can take as the coordinate charts collections of functions that are distance functions from certain points in the space. However, these 
coordinate maps are on measurable charts, and there may be more than one measurable chart needed to cover $\mu$-almost all of $X$.
The benefit of the differential structure is that when considering energy minimization on domains in $X$, with energy measured using
the Cheeger differential structure, the minimizer also satisfies a weak form of a PDE, namely the corresponding Euler-Lagrange equation;
such an equation is unavailable when minimizing upper gradient-based energy.
However, given the ambiguity of construction of the differential structure, a numerical approximation of solutions to the energy
minimization problem is difficult to compute.
Based on the work~\cite{ACDiM}, the paper~\cite{DS} gave a construction of a (nonlinear) Dirichlet-type form on a complete doubling
metric measure space supporting a $p$-Poincar\'e inequality, which does provide a way of numerically approximating
energy, see also~\cite{GL, BLS}.

This step in the process of constructing ways of numerically approximating 
energies in a metric measure space was to find a way of approximating 
energies via a discrete graph approximation of the metric measure space. This was undertaken in~\cite{BLS} using the tools of
$\Gamma$-convergence, with the limit energy also associated with a Dirichlet form. 
However, the tools of $\Gamma$-convergence does not suffice in obtaining energy minimizers on $X$ as
limits of energy minimizers on graphs. Traditional theory of Dirichlet forms ask for a stronger convergence of energy forms, called
\emph{Mosco convergence} as in~\cite{Mosco}, as for example in the works of Kuwae and Shioya (for instance, \cite{Shioya1, Shioya2}).
These works assume a uniform bound on either the Ricci curvature or the sectional curvature of the spaces, in particular that
it is an $RCD(K,N)$ space for some finite $N$, but such bounds
are generally unavailable even when the sequence metric spaces is equi-doubling and supports an equi-Poincar\'e inequality,
as seen in the case of the Heisenberg groups. \cite{Juillet2009} showed that the Heisenberg group 
equipped with the Carnot-Carath\'eodory distance does not satisfy $CD(K,N)$ for any 
$K,N$; hence it is not an $RCD(K,N)$ space. 
The goal of the present paper is to bypass the requirement
of Mosco convergence by considering a $\Gamma$-convergence of the graph energies with respect to the weak-* topology on
$N^{1,2}(X)$. We use this notion to establish convergence of graph-energy minimizers to energy minimizers in $X$. The following is the
main theorem of this paper.

\begin{thm}\label{thm:main}
Let $(X,d)$ be a complete metric space equipped with a doubling measure $\mu$ supporting a $2$-Poincar\'e inequality. 
Let $\Om$ be a bounded
domain in $X$ with $\mu(X\setminus\Om)>0$, and we fix $f\in N^{1,2}(X)$ as the function that serves as the Dirichlet boundary data for $\Om$.
For each $r>0$ let $X_r$ be a discretization of $X$ and
the discrete measure $\mu_r$ on $X_r$ a discretization of the measure $\mu$ on $X$ in the sense of~\cite{GL}. 
Consider the graph-approximation 
$\Om_r$ of $\Om$ in $X_r$. With $E_r$ a quadratic energy form on $D^{1,2}_0(\Om_r)$ as given by~\eqref{eq:Er-def}, there is an induced
energy form $\mathcal{E}_r$ on $N^{1,2}_0(\Om)$ such that for any positive sequence $(r_k)_k$  with
$\lim_kr_k=0$ we obtain a further subsequence $(r_{k_m})_m$ for which the following two conditions hold:
\begin{enumerate}
\item[{\bf (a)}] there exists a corresponding $\Gamma$-limit energy form $\mathcal{E}$ of $(\mathcal{E}_{r_{k_m}})_m$ with the limit energy
majorized by the upper gradient energy on $N^{1,2}(\Om)$,
\item[{\bf (b)}] there is a sequence $(u_m)$ of functions in $N^{1,2}(X)$, obtained from 
the graph 
energy minimizers $\uB[r_{k_m}]$ of 
$E_{r_{k_m}}$ on $\Om_{r_{k_m}}$, with boundary data $f_{r_{k_m}}$, that
weakly converges in $N^{1,2}(X)$ to a function $u_\infty\in N^{1,2}_0(\Om)$ such that $u_\infty+f$ is an $\mathcal{E}$-energy
minimizer on $\Om$ with boundary data $f$.
\end{enumerate}
\end{thm}

Thanks to the Rellich-Kondrachov embedding theorem, there is a subsequence of the sequence $u_m$ 
that converges in $L^p(X)$ to $u_\infty$, see for instance~\cite[Theorem~8.1]{HajKo}. Hence $u_m+f$ 
can be seen
as numerical approximations of the solution $u_\infty+f$.

The functions $u_m$ can be obtained from $\uB[r_{k_m}]$ as projections $P_{r_{k_m}}\uB[r_{k_m}]$, where
such projections $P_r$ map 
functions on $X_r$ to functions in $N^{1,2}(X)$. This is accomplished via the following theorem, which is also of independent interest.

\begin{thm}\label{thm:main-projection}
Let $(X,d)$ be a complete metric space equipped with a doubling measure $\mu$ supporting a $2$-Poincar\'e inequality, and let $\Om$ be a bounded
domain in $X$ with $\mu(X\setminus\Om)>0$. Let $X_r$ be a discretization of $X$ and
the discrete measure $\mu_r$ on $X_r$ a discretization of the measure $\mu$ on $X$ in the sense of~\cite{GL}.
Then there is a projection map $P_r:L^2(X_r)\to N^{1,2}(X)$ that satisfies the following three properties:
\begin{enumerate}
\item[{\bf (a)}] For each 
$u:X_r\to\R$ we have that for each $\xB\in X_r$,
\begin{equation} \label{proj_cond1Main}
    (P_ru)_r(\xB)=u(\xB).
\end{equation}
\item[{\bf (b)}] There is a constant $C\ge 1$ such that following lower bound estimate holds:
\begin{equation} \label{proj_cond2Main}
    \int_X g_{P_ru}^2\, d\mu
    \le C\, \sum_{\overline{x}\in X_r}\, \sum_{\overline{x}\sim \overline{y}\in X_r}\, \frac{|u(\overline{y})-u(\overline{x})|^2}{r^2}\, \mu_r(\{\overline{x}\})
\end{equation}
with the constant $C$ independent of $r$ and of $u$. 
\item[{\bf (c)}] The projection map also satisfies the following ``boundary consistency'':
\begin{equation} \label{proj_cond3Main}
    P_r:D^{1,2}_0(\Om_r)\to N^{1,2}_0(\Om).
\end{equation}
\end{enumerate}
\end{thm}

Indeed, in this paper we provide two distinct constructions of $P_r$, one using a Whitney type cover, 
and the other using an upper gradient-based McShane-type extension.

The approximating energies $\mathcal{E}_r$, $r>0$, considered in~\cite{BLS}, only give a weak control
in terms of the upper gradient energy; indeed, for $u\in N^{1,2}(X)$ we only have
\[
\int_Xg_u^2\, d\mu\lesssim \liminf_{r\to 0^+}\mathcal{E}_r[u].
\]
Here $g_u$ is the minimal $2$-weak upper gradient of $u$.
With the above constructions of the projection map $P_r$ we have a stronger lower bound control of
\emph{each} $\mathcal{E}_r[P_ru_r]$ in terms of $\int_X g_{P_ru_r}^2\, d\mu$, thanks to condition~(b) of Theorem~\ref{thm:main-projection}.

The structure of the paper is as follows. In the next section we describe the notions of upper gradient-based Sobolev-type spaces,
doubling property, and support of Poincar\'e inequalities, and the graph approximations of the metric measure space.
In Section~\ref{Sec:3} we provide two alternate constructions of the projection map $P_r$, and prove Theorem~\ref{thm:main-projection}
for each of these constructions.
Subsequently, in Section~\ref{Sec:4} we give a construction of a family of energy forms on $N^{1,2}_0(\Om)$ and 
study the $\Gamma$-limit of these energy forms when $N^{1,2}_0(\Om)$ is equipped with the weak topology. Finally,
in Section~\ref{Sec:5}, we establish that the weak limit of the projections of graph energy minimizers are minimizers of the
$\Gamma$-limit energy form, thus proving Theorem~\ref{thm:main}.

\section{Background}

We start with a metric measure space $(X,d,\mu)$ where $(X,d)$ is complete and $\mu$ is a doubling measure supporting a $2$-Poincar\'e 
inequality. By doubling measure we mean that there is a constant $C_D\ge 1$ such that
\[
\mu(B(x,2r))\le C_D\, \mu(B(x,r))
\]
whenever $x\in X$ and $r>0$.

\subsection{Newton-Sobolev spaces}

\begin{defn}
    Given a function $f:X\to\R$, we say that a non-negative Borel function $g:X\to[0,\infty]$ is an upper gradient of $f$ if
\[
|f(\gamma(b))-f(\gamma(a))|\le \int_\gamma g\, ds
\]
when $\gamma:[a,b]\to X$ is a rectifiable path (that is, a continuous map with finite length such that $\gamma$ is absolutely continuous).
\end{defn}

\begin{defn}
    We fix $p$ with $1<p<\infty$. We say that $f\in N^{1,p}(X)$ if the norm
    \[
    \|f\|_{N^{1,p}(X)}:=\left(\int_X|f|^p\, d\mu\right)^{1/p}+\inf_g\left(\int_Xg^p\, d\mu\right)^{1/p}
    \]
    is finite, where the infimum is over all upper gradients $g$ of $f$. 
\end{defn}

Given a set $E\subset X$, the $p$-capacity of $E$ is given by
\[
\text{Cap}_p(E):=\inf_f \left(\int_X|f|^p\, d\mu+\inf_g\, \int_Xg^p\, d\mu\right),
\]
where the first infimum is over all measurable functions $f$ with $f\ge 1$ on $E$, and the second infimum is over all
upper gradients $g$ of $f$. For properties related to the concepts defined above, we direct the interested
reader to~\cite{BBbook, HKSTbook}.

\begin{defn}
    We say that the metric measure space $(X,d,\mu)$ supports a $p$-Poincar\'e inequality if there are constants
    $C_P\ge 1$ and $\lambda\ge 1$ such that for measurable function -- upper gradient pairs $(f,g)$ we have
    \[
    \jint_B|f-f_B|\, d\mu\le C_P\, \rad(B)\, \left(\jint_{\lambda B}g^p\, d\mu\right)^{1/p}
    \]
    for all balls $B\subset X$. Here, $\lambda B$ is a ball that is concentric with $B$ but with radius $\lambda\, \rad(B)$.
\end{defn}

\begin{remark}\label{rem:geodesic}
If $(X,d,\mu)$ is locally complete, $\mu$ is doubling and supports a $p$-Poincar\'e inequality, then $X$ is a quasiconvex space; there
is a constant $C_Q\ge 1$ such that for each $x,y\in X$ there is a rectifiable curve $\gamma_{x,y}$ in $X$ with end points $x,y$ such that
the length $\ell(\gamma_{x,y})$ of $\gamma_{x,y}$ satisfies $\ell(\gamma_{x,y})\le C_Q\, d(x,y)$ (see for 
instance~\cite{Chee, HKSTbook}). In this case, the inner length metric
is bi-Lipschitz equivalent to the original metric on $X$, and so we can replace the original metric with the inner length metric. 
With respect to this new metric the space becomes a length space, and so by the results of Haj\l asz and Koskela, we can
assume that $\lambda=1$ at the expense of increasing the value of $C_P$. 
Moreover, as $(X,d)$ is complete and $\mu$ is a doubling measure supported on $X$, necessarily $(X,d)$ is 
proper, that is, closed and bounded subsets of $X$ are compact. Therefore, thanks to the Arzel\`a-Ascoli theorem, 
the length space $(X,d)$ is actually a geodesic space, that is, for each $x,y\in X$ there is a curve $\gamma$ in $X$ connecting $x$ and 
$y$. such that $\ell(\gamma)=d(x,y)$.
We refer the interested reader to~\cite{HKSTbook} for 
details concerning the claims in this paragraph. 
\end{remark}

We fix a bounded domain $\Om\subset X$ such that $\mu(X\setminus\Om)>0$, and we fix $f\in N^{1,2}(X)$ which will serve
as the boundary datum for the Dirichlet problem on $\Om$. The Banach space of interest to us is the space $N^{1,2}_0(\Om)$, to which we
build a bilinear energy form with boundary data $f$.

\begin{defn}
    A function $f\in N^{1,p}(X)$ is said to be in $N^{1,p}_0(\Om)$ if $f=0$ $p$-a.e.~in $X\setminus\Om$, that is, the set
    $E:=\{x\in X\setminus \Om\, :\, f(x)\ne 0\}$ satisfies $\text{Cap}_p(E)=0$.
\end{defn}

A consequence of the above Poincar\'e inequality is the following result, first formulated by Maz'ya~\cite{Mazya} in the Euclidean setting.
A proof of this lemma in the setting of metric measure spaces can be found in~\cite[Theorem 6.21]{BBbook}. A capacitary version of
this result, which is a much stronger version of this following lemma, is called the Maz'ya capacitary inequality, see the discussion
in~\cite{Mazya, BBbook}.

\begin{lem}\label{lem:Maz}
Suppose that $(X,d,\mu)$ is a complete metric measure space with $\mu$ a doubling measure supporting a $p$-Poincar\'e inequality. Let 
$\Om\subset X$ be a bounded domain such that $\mu(X\setminus\Om)>0$. Then there is a constant $C\ge 1$, that depends solely on
the doubling and Poincar\'e constants as well as the domain $\Om$, such that for every $u\in N^{1,p}_0(\Om)$ and upper gradient $g_u$ of $u$,
\[
\int_\Om|u|^p\, d\mu\le C\, \int_\Om g_u^p\, d\mu.
\]
\end{lem}

In this present note we will focus only on the case that $p=2$.

\subsection{Graph approximations of complete doubling metric measure spaces}

Given a positive real number $r$, a set $A\subset X$ is said to be $r$-separated if $d(x,y)\ge r$ whenever $x,y\in A$ with $x\ne y$.
We say that $A$ is a maximal $r$-separated set if adding another point to $A$ would result in $A$ being not $r$-separated, or equivalently,
$X=\bigcup_{x\in A}B(x,r)$ and $A$ is $r$-separated. An argument using Zorn's lemma in general metric setting, or more directly as in~\cite{HKSTbook}
when $X$ supports a doubling measure, shows that every $r$-separated subset of $X$ is a subset of a maximal $r$-separated
subset of $X$.

For $r>0$ let $X_r$ be a maximal $r$-separated subset of $X$, and is considered as the vertex set of the graph where, given 
two distinct vertices $\overline{x},\overline{y}\in X_r$ we say that $\overline{x}\sim\overline{y}$, that is, the two vertices are neighbors and
form an edge of the graph, if $d(\overline{x},\overline{y})\le 3r$.

For a function $u\in L^1_{loc}(X)$, we have a function $u_r:X_r\to\R$ given by
\[
u_r(\overline{x}):=\jint_{B(\overline{x},r/4)}u\, d\mu.
\]
When $u\in L^2(X)$ we have that $u_r\in L^2(X_r)$, where $X_r$ is equipped with the measure given by 
$\mu_r(\{\overline{x}\})=\mu(B(\overline{x},r/4))$.

It was shown in~\cite[Theorem~1.1, Theorem~1.2]{GL} that the metric spaces $X_r$, as described above, converge in the Gromov-Hausdorff
topology to $X$ as $r\to 0^+$. Moreover, the measure $\mu_r$ on $X_r$ converges in the measured Gromov-Hausdorff sense to
the measure $\mu$ on $X$. It was also shown there that when $(X,d,\mu)$ is doubling and supports a $p$-Poincar\'e inequality,
then the metric graphs $X_r$, equipped with the measure $\mu_r$, is also doubling and supports a $p$-Poincar\'e 
inequality, with the associated constants depending solely on the corresponding constants of $(X,d,\mu)$.

In \cite{BLS} we considered the forms
\begin{equation*}
    \mathcal{E}_r(u,v) = \sum_{\overline{x} \in X_r} \sum_{\overline{y} \sim \overline{x}} \frac{\left[u_r(\overline{y})-u_r(\overline{x})\right]\left[v_r(\overline{x})-v_r(\overline{y})\right]}{r^2} \mu_r(\{\xB\})
\end{equation*}
and showed that the associated quadratic functionals are comparable to the $2$-weak upper gradient energy if $(X,d,\mu)$ is doubling and supports a $2$-Poincar\'e inequality in the sense that there is a constant $C>0$ such that
$$
\frac{1}{C} \sup _{r>0} \mathcal{E}_r[u] \leq \int_X g_u^2 d \mu \leq C \liminf _{\varepsilon \rightarrow 0^{+}} \mathcal{E}_{\varepsilon}[u]
$$
Furthermore, it was shown that there is always a sequence $r_k\to 0^+$ such that  $\mathcal{E}_{r_k}$ $\Gamma$-converges to a Dirichlet form on $N^{1,2}(X)$, in the strong topology, on $N^{1,2}(X)$. However, we cannot guarantee that in the strong topology, the sequence of solutions of the Dirichlet boundary value problem associated with $\mathcal{E}_{r_k}$ converges. For this reason, in the present paper we reformulate the construction of quadratic energy forms using the weak topology on $N^{1,2}(X)$.   Indeed, as we are interested in approximations of solutions of  Dirichlet boundary value problems on a domain $\Omega\subset X$, with boundary data $f\in N^{1,2}(X)$, we focus on the Hilbert space $N^{1,2}_0(\Omega)$ in the latter part of the paper.

Given a function $u:X_r\to\R$, we set its graph energy by
\begin{equation}\label{eq:Er-def}
E_r(u):=\sum_{\overline{x}\in X_r}\, \sum_{\overline{x}\sim \overline{y}\in X_r}\, \frac{|(u+f_r)(\overline{y})-(u+f_r)(\overline{x})|^2}{r^2}\, \mu_r(\{\overline{x}\}).
\end{equation}
For each $\xB\in X_r$ we set
\begin{equation}\label{eq:graph-grad}
|\nabla_ru|(\xB):=\sum_{X_r\ni \overline{w}\sim \xB}\, \frac{|u(\xB)-u(\overline{w})|}{r}.
\end{equation}
For functions $u\in N^{1,2}_0(\Om)$, we set
\[
\widehat{\mathcal{E}_r} u:=E_r(u_r)
= \sum_{\overline{x}\in X_r}\, \sum_{\overline{x}\sim \overline{y}\in X_r}\, \frac{|(u_r+f_r)(\overline{y})-(u_r+f_r)(\overline{x})|^2}{r^2}\, 
\mu_r(\{\xB\}). 
\]
In the next section we consider projections of functions on $X_r$ to functions on $X$, and then we modify $\widehat{\mathcal{E}_r}$ further
to an energy form $\mathcal{E}_r$.

We now fix a bounded domain $\Om\subset X$ with $\mu(X\setminus\Om)>0$ as the domain in which we are interested in solving a 
Dirichlet boundary value problem. 
Let $\Om_r$ be a discretization of $\Om$ such that $\Om_r\subset X_r$, that is, $\Om_r$ consists of all $\xB\in X_r$ such that 
$B(\xB, 20r)\subset\Om$.
We set as $D^{1,2}_0(\Om_r)$ the collection of all functions
$u:X_r\to\R$ such that $u_r=0$ on $X_r\setminus\Om_r$.

\section{Proof of Theorem~\ref{thm:main-projection}: two constructions of the projection map} 
\label{Sec:3}

In this section, we want to construct a projection map $P_r:L^2(X_r)\to N^{1,2}(X)$ that satisfies the three properties
given in Theorem~\ref{thm:main-projection}. For the convenience of the reader, we repeat the three conditions here:
\begin{enumerate}
\item For each 
$u:X_r\to\R$ we have that for each $\xB\in X_r$,
\begin{equation} \label{proj_cond1}
    (P_ru)_r(\xB)=u(\xB).
\end{equation}
\item There is a constant $C\ge 1$ such that following lower bound estimate holds:
\begin{equation} \label{proj_cond2}
    \inf_g\int_X g^2\, d\mu\le C\, \sum_{\overline{x}\in X_r}\, \sum_{\overline{x}\sim \overline{y}\in X_r}\, \frac{|u(\overline{y})-u(\overline{x})|^2}{r^2}\, \mu_r(\{\overline{x}\}),
\end{equation}
where the infimum is over all upper gradients $g$ pf $P_ru$. 
 The constant $C$ is independent of $r$ (and of $u$). 
\item We also need the projection map to satisfy the following ``boundary consistency'':
\begin{equation} \label{proj_cond3}
    P_r:D^{1,2}_0(\Om_r)\to N^{1,2}_0(\Om).
\end{equation}
\end{enumerate}

In the following two subsections, we provide two possible constructions of such projection maps.

\subsection{Projection via Whitney decomposition}

The first construction of the projection $P_r$ utilizes a Whitney type decomposition of $X\setminus\bigcup_{\xB\in X_r}\overline{B}(\xB,r/4)$
and a corresponding partition of unity.

\begin{prop}[{\cite[Proposition~4.1.15 and page~104]{HKSTbook}}] \label{WhitneyFromNagesBook}
 Let $X=(X, d)$ be a doubling metric space with constant $N$ and let $O$ be an open subset of $X$ such that $X \backslash O \neq \emptyset$. There exists a countable collection $\mathcal{W}_{O}=\left\{B\left(x_i, r_i\right)\right\}_{i\in I\subset\N}$ of balls in $O$ such that
    \[
    r_i=\frac{1}{8} \operatorname{dist}\left(x_i, X \backslash O\right), 
    \]
    and 
    \[
    \Omega=\bigcup_i B\left(x_i, r_i\right), \qquad \sum_i \chi_{B\left(x_i, 2 r_i\right)} \leq 2 N^5.
    \]
    Moreover, there are constant $C\geq 1$ and nonnegative functions $\varphi_i$ such that 
    \begin{enumerate}
        \item[(i)] $\varphi_i(x)=0$ for $x \notin B\left(x_i, 2 r_i\right)$, and for every $x \in O$ we have that $\varphi_i(x) \neq 0$ 
        for at most $C$ indices $i$;
        \item[(ii)] $0 \leq \varphi_i \leq 1$ and $\varphi_i\vert_{B\left(x_i, r_i\right)} \geq C^{-1}$;
        \item[(iii)] $\varphi_i$ is $C / r_i$-Lipschitz;
        \item[(iv)]  $\sum_i \varphi_i(x)=1$ for every $x \in O$.
    \end{enumerate}
\end{prop}

Let $\mathcal{C}  = \bigcup_{\overline{x}\in X_r} \overline{B}(\overline{x}, r/4)$.
We now use a decomposition of $O=X\setminus\mathcal{C}$, $W_O=\{B(x_i,r_i)\}$ and  and the corresponding partition of unity $\{\varphi_i\}$ provided by the above Proposition applied to $O$.  
While $x_i$ denote the centers of balls covering $O$, by $\overline{x}_i$ we denote a fixed sequence of points in $\mathcal{C}\cap X_r$ such that $d(x_i,\mathcal{C}) = d(x_i,\overline{B}(\xB_i,r/4))$.

Now we can define the projection map $P_r:L^2(X_r)\to N^{1,2}(X)$ by putting for any $u\in L^2(X_r)$
\begin{equation} \label{Whitney_projection}
    P_r u(x) = \begin{cases}
        \sum_{i} u(\xB_i) \cdot \varphi_i(x), & x\in O\\
       u(\xB), & x\in \overline{B}(\xB, r/4), \text{ for some } \xB\in X_r
    \end{cases}
\end{equation}

It is easy to check that due to properties $(i)-(iv)$ above function $P_r u$ is well defined and condition \eqref{proj_cond1} holds; we leave the details to the
reader.

\begin{lem}\label{lem_liplip}
    $P_r u$ is a locally Lipschitz function and there exists $C>0$ such that 
     \[
     |P_r u(x) - P_r u(y)| \leq C  \left( \sup_{\overline{w}\in X_r\cap B(x,3r)} |\nabla_r u|(\overline{w}) \right) \cdot d(x,y)
     \]
     for all $u\in L^2(X_r,\mu_r)$ and for all $x,y\in X$ with $d(x,y)<r$. 
\end{lem}

\begin{proof}
    To show  the Lipschitz property, we break down the argument into three cases based on whether $x\in X$ and $y\in B(x,r)$ are in 
    $O$ or its complement $\mathcal{C}$. 
\vskip .5cm

\noindent{\bf Case 1:} Suppose that $x, y\in\mathcal{C}$. 
    In this case there are $\overline{x},\overline{y}\in X_r$ such that $x\in \overline{B}(\xB,r/4)$ and $y\in\overline{B}(\overline{y},r/4)$.
    If $\xB=\overline{y}$, then there is nothing to prove as 
    \begin{equation*}
        |P_ru(x)-P_ru(y)|=0\le |\nabla_ru(\xB)|\, d(x,y).
    \end{equation*}

    If $\xB\ne \overline{y}$, then $d(\overline{x},\overline{y})>r$, so by triangle inequality $\frac{r}{2}=r-\frac{r}{4}-\frac{r}{4}\le d(x,y)$. Also since $d(x,y)<r$ the same triangle inequality implies $d(\xB,\overline{y})\le r+\frac{r}{4}+\frac{r}{4}=\frac{3r}{2}$. From these two facts we see that  $\xB\sim\overline{y}$ and 
    \[
        |P_r u(x) - P_r u(y)| = |u(\overline{x}) - u(\overline{y})|\le |\nabla_ru(\xB)|\, r\le 2\, |\nabla_ru(\xB)|\, d(x,y).
    \] 

\vskip .5cm    
    
\noindent{\bf Case 2:} $x\in O$ and $y\in \mathcal{C}$.  In this case, there is a point 
$\overline{y}\in X_r$ such that $y\in\overline{B}(\overline{y},r/4)$, and $P_r u(y) = u(\overline{y})$. As $\varphi_i$ is a partition of unity we can write 
    \[
   | P_r u(x)  - P_r u(y)|  = \left| \sum_i (u(\overline{x}_i) - u(\overline{y})) \varphi_i(x) \right|.
    \]
    By property (i) in Proposition 3.4 no more than $C$ of them are nonzero at $x$ and for all of the corresponding indexes $i$ we have $d(x,x_i)<2r_i$. In other words, 
    \begin{equation}
        \label{OC}
        | P_r u(x)  - P_r u(y)|  \leq C \max_{i\, :\, d(x_i,x)\leq 2r_i}\ |u(\xB_i) - u(\yB)|.
    \end{equation}
    All such $r_i$ are comparable as $d(x_i,\mathcal{C}) = 8r_i$ and thus 
    \begin{equation*}
        6r_i \leq d(x,\mathcal{C})\leq 10 r_i.
    \end{equation*}
    Note that if $d(x,y) < \frac{r}{8}  $ then $d(x,\mathcal{C})< \frac{r}{8} $ and thanks to the above estimate 
    \begin{equation*}
       d(x_i,B(\yB, r/4)) \leq d(x_i,x) + d(x,\mathcal{C}) \leq 2r_i + d(x,\mathcal{C}) \leq \frac{4}{3}  d(x,\mathcal{C}) < \frac{r}{6}. 
    \end{equation*}
    As all elements of $X_r$ are $r$-separated, it shows that if $d(x,y) < \frac{r}{8}  $ then all $\overline{x}_i$ in \eqref{OC} have to be $\overline{y}$, and there is nothing for us to prove. 
    
    The last conclusion allows us to assume that 
    \begin{equation}
        \label{rEst}
        d(x,y)\geq \frac{r}{8}>r_i.
    \end{equation}
    Then 
    \[
    d(x,y)<r, \ d(x_i,x)<r/4, \ d(x_i,\overline{x}_i)\leq r/4+8r_i< 5r/4, \ d(y,\overline{y})\leq r/4
    \] 
    and the triangle inequality yields
    \begin{equation*}
        d(\overline{x}_i,\overline{y}) < 3r,
    \end{equation*}
    which makes $\overline{x}_i\sim \overline{y}$ for all $\overline{x}_i$ in \eqref{OC}, and hence 
    \[
        | P_r u(x)  - P_r u(y)|  \leq C |\nabla_r u|(\yB) \cdot r.
    \]
    Recalling \eqref{rEst} we get 
    \begin{equation*}
        | P_r u(x)  - P_r u(y)|  \leq 8C |\nabla_r u|(\yB) \cdot d(x,y). 
    \end{equation*}

    \noindent\textbf{Case 3:} Both $x\in O$ and  $y\in O$. 
    There are two possibilities - either both $x$ and $y$ lie in the same ball $B(x_j,2r_j)$ for some $j$ or 
    $y\in O\setminus \bigcup_{2B_k\ni x} 2B_k$. In the former case,  using the fact that $\varphi_i$ are $C/r_i$-Lipschitz 
    \begin{equation*}
        | P_r u(x)  - P_r u(y)| =  \left| \sum_i u(\xB_i) (\varphi_i(x) - \varphi_i(y)  )\right|  = \left| \sum_i (u(\xB_i) - u(\xB_j)  )(\varphi_i(x) - \varphi_i(y)  )\right| 
    \end{equation*}
    \begin{equation*}
         \leq C \max_{d(x_i,x_j)\leq 2r_i} \frac{|u(\xB_i) - u(\xB_j) |}{r_i} \cdot d(x,y) \leq 100C |\nabla_r u(\xB_j)| d(x,y),
    \end{equation*}
    where the last inequality can be justified as in the above case assuming that  $r_j\leq r/100$, and arguing that $r_j> r/100$ would imply that all $\xB_i = \xB_j$.

    It now remains to prove the result assuming $y\in O\setminus \bigcup_{2B_k\ni x} 2B_k$.  Fix any two balls $B(x_i,r_i), B(x_j,r_j)$ that contain $x$ and $y$. Also put $x'\in B(\xB_i, r/4) \cap B(x, 9 r_i)$ and $y'\in B(\yB_j, r/4) \cap B(y, 9 r_j)$. Then $ d(x,y) \geq  \max(r_i,r_j)$ and therefore 
    \begin{equation*}
        d(x,x') < 9r_i < 9 d(x,y),
    \end{equation*}
    \begin{equation*}
        d(y,y') < 9r_j < 9 d(x,y),
    \end{equation*}
    \begin{equation*}
        d(x',y') < 19 d(x,y),
    \end{equation*}
    and by the previous cases
    \begin{align*}
        |P_r u(x) - P_r u(y)| &\leq |P_r u(x) - P_r u(x')| + |P_r u(x') - P_r u(y')| + |P_r u(y') - P_r u(y)| \\
       & \leq C |\nabla_r u(\xB_i)| d(x,x') + C |\nabla_r u(\xB_i)| d(x',y') +  C |\nabla_r u|(\yB_j) d(y',y)\\
       & < 37 C \max_{w} |\nabla_r u(w)| d(x,y).
    \end{align*}
    \end{proof}

\begin{lem}
    Function $P_r u$ satisfies conditions \eqref{proj_cond2} and \eqref{proj_cond3}.    
\end{lem}

\begin{proof}
Lemma \ref{lem_liplip} allows us to show \eqref{proj_cond2} as follows. Note that for any rectifiable $\gamma:[a,b]\to X$, and any partition $t_i\in [a,b]$  we have 
\begin{equation*}
    |P_r u(\gamma(b)) - P_r u(\gamma(a))| \leq \sum_{i=1}^{N} |P_r u(\gamma(t_i)) - P_r u(\gamma(t_{i-1}))|.
\end{equation*}
Assuming that $t_i$ are chosen so that the arclength of $\gamma\vert_{[t_{i-1},t_i]}$ is less than $r$ we get from Lemma \ref{lem_liplip} that 
\begin{equation} \label{t2}
    \begin{aligned}
        |P_r u(\gamma(b)) - P_r u(\gamma(a))| & \leq \sum_{i=1}^{N} C  \left( \sup_{\overline{w}\in X_r\cap B(\gamma(t_i),3r)} |\nabla_r u|(\overline{w}) \right)  d(\gamma(t_i),\gamma(t_{i-1}) ) \\
        & \leq  \sum_{i=1}^{N} \int_{\gamma\vert_{[t_{i-1},t_i]}} C \sup_{\overline{w}\in X_r\cap B(\gamma(t),4r)} |\nabla_r u|(\overline{w})  \ ds.  \\
        & = \int_{\gamma} C \sup_{\overline{w}\in X_r\cap B(\gamma(t),4r)} |\nabla_r u|(\overline{w})  \ ds
    \end{aligned}
\end{equation}
which shows that $C \sup_{\overline{w}\in X_r\cap B(\gamma(t),4r)} |\nabla_r u|(\overline{w}) $ is an upper gradient of $P_r u$.

Denoting by $\chi:\R\to\R$ the characteristic function of interval $[0,1]$ we also get 
\begin{equation}
    \label{t3}
    \begin{aligned}
        \int\limits_X \sup_{\overline{w}\in X_r\cap B(x,4r)} |\nabla_r u|^2(\overline{w}) \ d\mu(x) & \leq  \int\limits_X \sum_{\overline{w}\in X_r} \chi\left(\frac{d(x,\overline{w})}{4r}\right) |\nabla_r u|^2(\overline{w}) \ d\mu(x) \\
        &  =  \sum_{\overline{w}\in X_r} |\nabla_r u|^2(\overline{w})\int\limits_X  \chi\left(\frac{d(x,\overline{w})}{4r}\right)  \ d\mu(x) \\ 
        &  = \sum_{\overline{w}\in X_r} |\nabla_r u|^2(\overline{w}) \cdot \mu(B(\overline{w}, 4r))\\
        & \leq C \sum_{\overline{w}\in X_r} |\nabla_r u|^2(\overline{w})\, \mu_r(\{\overline{w}\}), 
    \end{aligned}
\end{equation}
for any $C$ greater than the square of the doubling constant.

From \eqref{t2} and \eqref{t3} we deduce that there is $C>0$ such that 
\begin{equation*}
    \inf_{g} \int_X g^2\, d\mu\le   C \sum_{\overline{w}\in X_r} |\nabla_r u|^2(\overline{w})\, \mu_r(\{\overline{w}\}),
\end{equation*}
where the infimum is over all upper gradients $g$ of $P_ru$.
i.e. \eqref{proj_cond2} holds.

Finally, to see that \eqref{proj_cond3} holds, suppose  that $u\in D^{1,2}_0(\Omega_r)$ i.e.
\begin{equation} \label{D12_0_cond_backwards}
    B(\xB, 10r)\not\subset \Omega \implies  u(\xB) =  0, \qquad \forall \xB\in X_r
\end{equation} 
and consider any $x\in X\setminus \Omega$. Then either $x\in \cup_{\xB\in X_r} \overline{B}(\xB, r/4)$ or not. 

If there is $\xB\in X_r$ such that $x \in \overline{B}(\xB, r/4)$, then $d(\xB,X\setminus \Omega)\leq r/4$ and from to \eqref{Whitney_projection} and \eqref{D12_0_cond_backwards}
\begin{equation} \label{t4}
    P_r u(x) = u(\xB) = 0.
\end{equation}

If $x\notin \cup_{\xB\in X_r} \overline{B}(\xB, r/4)$, then by 
Proposition~\ref{WhitneyFromNagesBook} and the construction of $P_r$ there is $x_j\in X\setminus \mathcal{C}$ such that 
\begin{equation*}
    x\in B(x_j, r_j) \cap (X\setminus \Omega)
\end{equation*}
and 
\begin{equation} \label{t5}
    P_r u(x) = \sum_{i} u(\xB_i) \varphi_i(x), 
\end{equation}
where 
\begin{equation} \label{t6}
    d(x,x_i)<2r_i < r/4, \qquad \text{ and } \qquad  d(x_i, B(\xB_i, r/4)) = 8r_i < r.
\end{equation}
As $x\in X\setminus \Omega$, and \eqref{t6} shows that $  d(\xB_i,x) < 2r$ we conclude from \eqref{t5} and \eqref{D12_0_cond_backwards} that 
\begin{equation} \label{t7}
    P_r u(x) = 0.
\end{equation}
Lines \eqref{t4} and \eqref{t7} now show that $P_r u \in N^{1,2}_0(\Omega)$ i.e. \eqref{proj_cond3} holds. 
\end{proof}

\subsection{Projection via path-integral extensions}

In this second construction of the projection map $P_r$, we use a variant of the McShane extension technique usually
employed for Lipschtiz functions; see~\cite{Chee}.

\begin{defn}
For any $r>0$ and any $u:X_r \longrightarrow \mathbb{R}$, we define 
\begin{align}
\overline{g}(x)=\ 2\, \sum_{\overline{w}\in X_r,d(\overline{w},x)<3r} \sum_{\overline{y}\sim \overline{w}} \frac{\vert u(\overline{y})-u(\overline{w})\vert}{r},
\end{align}
and for $x\in X$, set
\begin{align}
P_ru(x)=\inf\left\{u(\xB)+\int_{\gamma}\overline{g}ds\right\},
\end{align}
where the infimum is taken over all $\xB\in X_r$ and all rectifiable curves joining $x$ and $\overline{B}(\xB,\tfrac{r}{4})$.
\end{defn}

The fact that $P_ru$ is measurable on $X$ follows from~\cite[Theorem~9.3.1]{HKSTbook}.

\begin{thm}\label{thm:gbar-ug}
There is a constant $C_1>0$ such that the following holds true.
For any $r>0$ and any $u:X_r \longrightarrow \mathbb{R}$, the function $\overline{g}$ is an upper gradient of $P_ru$ and
we have
\begin{align} \label{ineq.PrGradient.UpperBound}
\inf_g\, \int_X g^2\, d\mu\le C_1 \sum_{\xB\in X_r}\sum_{\yB\sim\xB}\, \frac{|u(\xB)-u(\yB)|^2}{r^2}\mu_r(\{\xB\}),
\end{align}
where the infimum is over all upper gradients $g$ of $P_ru$. That is, \eqref{proj_cond2} holds true.
\end{thm}

\begin{proof}
We first show that $\overline{g}$ is an upper gradient of $P_ru$. Adopting the idea from the proof 
of~\cite[Lemma 7.2.13]{HKSTbook}, we need to show that
\begin{align} \label{ineq.UpperGradient}
\vert P_ru(x)-P_ru(y)\vert \leqslant \int_{\gamma} \overline{g}ds
\end{align}
for any rectifiable curve $\gamma$ joining $x$ and $y$. Let us fix a rectifiable curve $\gamma$ in $X$ joining $x$ to $y$.

For any $\varepsilon>0$, there exists some $\xB \in X_r$ such that 
\begin{align}
P_ru(x) \geqslant u(\xB)+\int_{\gamma_{\varepsilon}}\overline{g}ds-\varepsilon
\end{align}
for some rectifiable curve $\gamma_{\varepsilon}$ joining $x$ and some point $z\in \overline{B}(\xB,\frac{r}{4})$. The concatenation
of $\gamma_\eps$ with $\gamma$ gives a rectifiable curve joining $y$ to $\overline{B}(\xB,\tfrac{r}{4})$, and so
\[
P_ru(y)\le u(\xB)+\int_{\gamma_\eps}\overline{g}\, ds+\int_\gamma\overline{g}\, ds
\le P_ru(x)+\eps+\int_\gamma\overline{g}\, ds.
\]
Letting $\eps\to 0^+$ we see that
\[
P_ru(y)\le P_ru(x)+\int_\gamma\overline{g}\, ds.
\]
Reversing the roles of $x$ and $y$ in the above argument also gives
\[
P_ru(x)\le P_ru(y)+\int_\gamma\overline{g}\, ds.
\]
Therefore~\eqref{ineq.UpperGradient} holds,
that is, $\overline{g}$ is an upper gradient of $P_ru$.

Further, by the doubling property of $\mu$ (which implies that the degree of each vertex has a uniform upper bound),
\begin{align*}
 \int_{B(\xB,4r)}\overline{g}^2\, d\mu 
&\le C\, \mu(B(\xB,r))\, \sum_{\overline{w}\in X_r\, :\, d(\overline{w},\xB)<4r}\ \sum_{\yB\sim\overline{w}}\frac{|\uB_r(\yB)-\uB_r(\overline{w})|^2}{r^2}\\
&\le C\,  \sum_{\overline{w}\in X_r\, :\, d(\overline{w},\xB)<4r}\ 
 \sum_{\yB\sim\overline{w}}\frac{|\uB_r(\yB)-\uB_r(\overline{w})|^2}{r^2}\, \mu_r(\{\overline{w}\}).
\end{align*}
By the bounded overlap property of the balls $B(\xB,4r)$, $\xB\in X_r$, and by the fact that $B(\xB,r)$, $\xB\in X_r$,
covers $X$, we obtain the desired inequality
 \[
\inf_g\ \int_Xg^2\, d\mu\le  \int_{X}\overline{g}^2\, d\mu
\lesssim \sum_{\overline{w}\in X_r}\sum_{\yB\sim\overline{w}}\frac{|\uB_r(\yB)-\uB_r(\overline{w})|^2}{r^2}\, \mu_r(\{\overline{w}\}),
\]
where the infimum is over all upper gradients of $P_ru$.
\end{proof}

\begin{lem}\label{lem:Prz}
For any $z\in B(\xB,\frac{r}{4})$, we have $(P_ru)(z)=u(\xB)$ and thus $(P_ru)_r=u$, which shows that~\eqref{proj_cond1} holds true.
\end{lem}

\begin{proof}
It suffices to show $(P_ru)(z)=u(\xB)$ for every $z\in \overline{B}(\xB,\frac{r}{4})$, and then by the definition of $(P_ru)_r$, we obtain $(P_ru)_r=u$.

Let $z\in \overline{B}(\xB,\tfrac{r}{4})$. 
By the construction of $P_ru$ we know that $P_ru(z)\le u(\xB)$.
So it suffices to show that $(P_ru)(z) \ge u(\xB)$.  
To do so,  we only need to show that
\[
u(\xB)-u(\zB) \leqslant \int_{\gamma} \overline{g} ds
\] 
for all 
$\zB\in X_r$ and any rectifiable curve $\gamma:[0,1] \rightarrow X$ joining $z$ and some point $y\in \overline{B}(\zB,\frac{r}{4})$. 
Let us now fix $\zB\in X_r$ with $\zB\ne \xB$, and $\gamma$ a rectifiable curve as above, connecting $z$ to some 
point $y\in\overline{B}(\zB,\tfrac{r}{4})$.

If $\gamma \subseteq B(\xB,\frac{3r}{2})$, then we have $B(\xB,\frac{3r}{2}) \cap B(\zB,\frac{r}{4}) \neq \emptyset$, so 
$d(\xB,\zB) \le \frac{3r}{2}+\frac{r}{4}=\frac{7r}{4} <3r$ and thus $\xB \sim \zB$. As $d(\xB,\zB) \ge r$, we have
\begin{align*}
\ell(\gamma) \ge d(\xB,\zB)-\frac{r}{4}-\frac{r}{4} \ge \frac{r}{2}.
\end{align*}
For any $x\in \gamma$, we have $d(x,\xB) \le \frac{3r}{2}<3r$ automatically, since $\gamma \subseteq B(\xB,\frac{3r}{2})$. This gives 
\[
2\, \frac{\vert  u(\zB)-u(\xB) \vert}{r} \le \overline{g}(x)
\] 
for any $x\in \widetilde{\gamma}$. In this way, we have, as desired, 
\begin{align*}
u(\xB)-u(\zB)  \le \vert  u(\zB)-u(\xB) \vert\le 2\, \frac{\vert  u(\zB)-u(\xB) \vert}{r} \cdot \frac{r}{2}
\le \int_{\gamma} \overline{g} ds.
\end{align*}
If $\gamma$ leaves $B(\xB,\frac{3r}{2})$, then we construct a chain of balls $\{B(\xB_i,r) \}_{i=1}^N$ connecting 
$\xB$ and $\zB$ as follows. Let $\xB_0=\xB$ and $t_0=0$, and we then set
\[
t_1=\inf\{t\in [0,1]\, :\, \gamma(t)\not\in B(\xB_0,\tfrac{3r}{2})\}.
\] 
Note that $\gamma(t_1)\in\overline{B}(\xB,\frac{3r}{2})$.
Assuming that $\xB_i$ and $t_i$, $i=0,\cdots, j$ have been chosen so that $\gamma([t_i,t_{i+1}))\subset B(\xB_i,\frac{3r}{2})$ for
$i=0,\cdots, j-1$, then we have either 
$t_j=1$, in which case the choosing procedure terminates, or we have 
$\gamma([t_{j},1])\subset \overline{B}(\xB_j, \tfrac{3r}{2})$, in which case the process of choosing the points
$\xB_i$ and $t_i$ stops as well, or else we choose 
\[
t_{j+1}=\inf\{t\in [t_j,1]\, :\, \gamma(t)\not\in B(\xB_i,\tfrac{3r}{2})\}
\]
and then choose $\xB_{j+1}\in X_r$ such that $d(\gamma(t_{j+1}),\xB_{j+1})<r$. 
Note that $\gamma(t_{j+1})\in\overline{B}(\xB_j,\frac{3r}{2})\setminus B(\xB_j,\frac{3r}{2})$.

This process eventually terminates as $\gamma$ is rectifiable, and $\ell(\gamma\vert_{[t_i,t_{i+1}]})\ge r/2$ at each intermediate 
step above because $\gamma\vert_{[t_i,t_{i+1}]}$ intersects both $\overline{B}(\xB_i,\tfrac{3r}{2})$ 
and $X\setminus B(\xB_i,\tfrac{3r}{2})$. So there is some
positive integer $N$ for which $\gamma([t_{N},1])\subset \overline{B}(\xB_N,\frac{3r}{2})$.

By our construction of 
$\xB_i$ and $\xB_{i+1}$, we see that $r\le d(\xB_i,\xB_{i+1}) \le r+\tfrac{3r}{2}<2r$  
for $i=0,1,\cdots,N-1$. 
Thus $\xB_i \sim \xB_{i+1}$ for $i=0,1,\cdots,N-1$. 
For $i=0,\cdots, N-1$ set $\gamma_i=\gamma\vert_{[t_{i},t_{i+1}]}$; as pointed out above, $\ell(\gamma_i)\ge \frac{r}{2}$.
For any $z\in \gamma_i$, we have $d(z,\xB_i) \leqslant \frac{3r}{2}< 3r$ automatically, 
since $\gamma_i \subseteq \overline{B}(\xB_i,\frac{3r}{2})$. This gives 
\[
2\, \frac{\vert u(\xB_i)-u(\xB_{i+1})\vert}{r} \le \overline{g}(z)
\] 
for any point $z$ in the trajectory of $\gamma_i$. In this way, for each $i=0,1,\cdots,N-2$,
\begin{align*}
\vert u(\xB_i)-u(\xB_{i+1})\vert  \le 2\, \frac{\vert u(\xB_i)-u(\xB_{i+1})\vert}{r} \cdot \frac{r}{2}
\le \int_{\gamma_i} \overline{g} ds.
\end{align*}
It follows that for $j=1,\cdots, m\le N-1$, 
\[
u(\xB)-u(\xB_{j+1})\le |u(\xB)-u(\xB_{j+1})|\le \sum_{i=0}^{j}|u(\xB_i)-u(\xB_{i+1})|\le \int_{\gamma\vert_{[0,t_{j+1}]}}\overline{g}\, ds.
\]
If $\zB=\xB_j$ for some $j\in\{0,\cdots, N\}$, then we have that $u(\xB)\le u(\zB)+\int_\gamma\overline{g}\, ds$ as desired.
If there is no such $j$, then in particular we have that $\xB_{N-1}\ne \zB$; however, $\gamma([t_N,1])$ is contained in
$B(\xB_N, \frac{3r}{2})$ and $\gamma(1)\in\overline{B}(\zB,r/4)$. It follows that 
\[
d(\xB_N,\zB)\le d(\xB_N, \gamma(1))+d(\gamma(1),\zB)<\frac{3r}{2}+\frac{r}{4}<3r,
\]
and so $\zB\sim\xB_N$. We set
$t_{N+1}=1$ here. As in the above, $\gamma_N=\gamma\vert_{[t_N,t_{N+1}]}$. 
If $\ell(\gamma_N)\ge \tfrac{r}{2}$, then as in the argument handling $\gamma_j$ for $j\le N-1$ also gives
\[
|u(\xB_N)-u(\zB)|\le \int_{\gamma_N}\overline{g}\, ds,
\]
from which we will have 
\[
u(\xB)-u(\zB)\le |u(\xB)-u(\xB_N)|+|(\xB_N)-u(\zB)|\le \int_{\gamma}\overline{g}\, ds.
\]
If $\ell(\gamma_N)<\tfrac{r}{2}$, then $d(\xB_{N-1},\zB) \le d(\xB_{N-1},\gamma(t_{N}))+\ell(\gamma_N)+d(\gamma(1),\zB) <\tfrac{3r}{2}+\tfrac{r}{2}+\tfrac{r}{4}=\tfrac{9r}{4}<3r$, which forces $\xB_{N-1}\sim \zB$, and so
\[
\overline{g}\vert_{\gamma_{N-1}}\ge 2\, \frac{|u(\xB_{N-1})-u(\zB)|}{r}.
\]
In this case, we have
\begin{align*}
u(\xB)-u(\zB)\le \sum_{i=0}^{N-2}|u(\xB_i)-u(\xB_{i+1})| +|u(\xB_{N-1})-u(\zB)| \le \int_{\gamma\vert_{[0,t_{N-1}]}}\overline{g}\, ds \le \int_{\gamma} \overline{g}\, ds.
\end{align*}

Combining the above arguments, we see that 
$u(\xB)\le u(\zB)+\int_\gamma\overline{g}\, ds$ whenever $\gamma$ is a rectifiable
path in $X$ connecting $\overline{B}(\xB,r/4)$ with $\overline{B}(\zB,r/4)$. Therefore $P_ru(z)=u(\xB)$ as claimed.
\end{proof}

Recall that with $\Om$ a bounded domain in $X$ with $\mu(X\setminus\Om)>0$, we have $\Om_r$ to be the collection of
all $\xB\in X_r$ for which $B(\xB, 20r)\subset\Om$. 
The analog of the following lemma for the first construction of $P_r$ follows immediately from its construction.

\begin{lem}
Suppose that $u:X_r\to\R$ such that $u=0$ on $X_r\setminus\Om_r$. Then $P_ru=0$ on $X\setminus\Om$, that is,
$P_ru\in N^{1,2}_0(\Om)$. Hence~\eqref{proj_cond3} holds true.
\end{lem}

\begin{proof}
We set $\Om[r]=\{x\in\Om\, :\, B(x,r)\subset\Om\}$.
From the construction of $\overline{g}$, we know that for $x\in X$,
\[
\overline{g}(x)=2\, \sum_{X_r\ni\overline{w}\, :\, d(x,\overline{w})<4r}\ \sum_{\zB\sim\overline{w}}\ \frac{|u(\zB)-u(\overline{w})|}{r},
\]
and so $\overline{g}=0$ on $X\setminus\Om[15r]$. From Theorem~\ref{thm:gbar-ug} we also know that $\overline{g}$ is an
upper gradient of $P_ru$, and from Lemma~\ref{lem:Prz} we know that for each $\zB\in X_r\setminus\Om_r$ we have
$P_ru=0$ on the closed ball $\overline{B}(\zB,r/4)$.

Let $x\in X\setminus\Om$. Then there is some $\xB\in X_r$ such that $x\in B(\xB,r)$. 
Recall from Remark~\ref{rem:geodesic} that $(X,d)$ is a geodesic space.
Let $\beta$ be a geodesic in $X$
with end points $x$ and $\xB$. Then the trajectory of $\beta$ lies entirely in $X\setminus\Om[15r]$, and so
\[
|P_ru(x)|=|P_ru(x)-P_ru(\xB)|\le \int_\beta\overline{g}\, ds=0,
\]
from which it follows that $P_ru(x)=0$. This completes the proof.
\end{proof}

\section{$\Gamma$-convergence under weak topology}\label{Sec:4}

We have now fixed $f\in N^{1,2}(X)$ and $\Om$ a bounded domain such that $\mu(X\setminus \Om)>0$.
We have seen that for $u:X_r\to\R$ with $u=0$ on $X_r\setminus\Om_r$, we have $P_ru\in N^{1,2}_0(\Om)$.
We now work with the Banach space $N^{1,2}_0(\Om)$, equipped with the weak topology.

Recall that for $r>0$, the space $D^{1,2}_0(\Om_r)$ is the collection of all functions $u:X_r\to\R$ for which
$u=0$ on $X_r\setminus\Om_r$, and let
$\mathcal{E}_r:N^{1,2}_0(\Om)\to[0,\infty]$ by setting
\[
\mathcal{E}_r(v)=\begin{cases}\widehat{\mathcal{E}_r}(u)=E_r(u_r) &\text{ if }v=P_ru \text{ for some }u\in D^{1,2}_0(\Om_r), \\ 
   \infty&\text{ if no such }u\text{ exists. }\end{cases}
\]

The following definition is from~\cite{DalMaso}.

\begin{defn}
    Let $X$ be a topological space.  Let $\left(F_k\right)$ be a sequence of functions from $X$ into $\overline{\mathbf{R}}$. We say that the 
$\Gamma$-limit of $\left(F_k\right)$ exists if 
 \[
    \sup _{U \in \mathcal{N}(v)} \liminf _{k \rightarrow \infty} \inf _{w \in U} F_k(w) 
      = \sup _{U \in \mathcal{N}(v)} \limsup _{k\rightarrow \infty} \inf _{w \in U} F_k(w) , 
 \]
    where $\mathcal{N}(v)$ denotes the set of all open neighbourhoods of $v$ in $X$. Whenever $\Gamma$-limit exists we denote it by 
 $\Gamma$-$\lim\limits_{k\to \infty} F_k(v)$.
\end{defn}

We wish to use the following~\cite[Corollary~8.12]{DalMaso} to construct $\Gamma$-limits of a sequence $\mathcal{E}_{r_i}$
with $\lim_ir_i=0$.

\begin{lem}\label{lem:DalMaso}
Assume that $\mathcal{U}$ is a Banach space with a separable dual. 
For each positive integer $k$ suppose that $F_k:\mathcal{U}\to[-\infty,\infty]$, and 
let $\Psi: \mathcal{U} \rightarrow \overline{\mathbf{R}}$ be a function satisfying 
\begin{equation}\label{eq:DalMasoForce}
\lim _{\|u\| \rightarrow+\infty} \Psi(u)=+\infty.
\end{equation}
If $F_k \geq \Psi$ for every $k \in \mathbf{N}$, then there exists a subsequence of 
$\left(F_k\right)$ which $\Gamma$-converges in the weak topology of $\mathcal{U}$.
\end{lem}

Note that $\mathcal{U}=N^{1,2}_0(\Om)$ is a reflexive Banach space with a separable dual, see Remark~\ref{rem:sequentialDalMaso} below 
and~\cite[Lemma 2.4.3]{HKSTbook}. 
The choice of 
\begin{equation}\label{eq:Def-Psi}
\Psi(u)=\left(\left(\frac{1}{C_1}\inf_g\, \int_\Om g^2\, d\mu\right)^{1/2}-\, \left(C_2\, \inf_{g_f}\, \int_\Om g_f^2\, d\mu\right)^{1/2}\right)_+^2,
\end{equation}
where the first infimum is over all upper gradients $g$ of $u$ and the second infimum is over all upper gradients $g_f$ of $u$, and
with $C_1$ given by~\eqref{ineq.PrGradient.UpperBound} and $C_2$ the constant from~\cite[Theorem~1.1(1)]{BLS},
satisfies the inequality~\eqref{eq:DalMasoForce}.
The inequality~\eqref{eq:DalMasoForce} follows from the Maz'ya inequality as in Lemma~\ref{lem:Maz} stated above.

\begin{lem}\label{lemPsiE}
For each $r>0$ and $v\in N^{1,2}_0(\Om)$ we have $\mathcal{E}_r(v)\ge \Psi(v)$, and hence $\mathcal{E}_r$ is coercive.
\end{lem}

\begin{proof}
If there is no $u\in D^{1,2}_0(\Om_r)$ such that $v=P_ru$, then the claimed inequality is trivially true. So suppose that $v=P_ru$
for some $u\in D^{1,2}_0(\Om_r)$. 
Note that $(w+h)_r=w_r+h_r$ whenever $w,h\in L^2_{loc}(X)$. Then
\begin{align*}
\mathcal{E}_r&(v)^{1/2}=\left(\sum_{\xB\in X_r}\, \sum_{\xB\sim\overline{y}\in X_r}\, \frac{|(u+f_r)(\yB)-(u+f_r)(\xB)|^2}{r^2}\mu_r(\{\xB\})\right)^{1/2}\\
\ge &\left(\left(\sum_{\xB\in X_r}\, \sum_{\xB\sim\overline{y}\in X_r}\, \frac{|u(\yB)-u(\xB)|^2}{r^2}\mu_r(\{\xB\})\right)^{1/2}
  - \left(\sum_{\xB\in X_r}\, \sum_{\xB\sim\overline{y}\in X_r}\, \frac{|f_r(\yB)-f_r(\xB)|^2}{r^2}\mu_r(\{\xB\})\right)^{1/2}\right)_+\\
 \ge & \left(\left(\sum_{\xB\in X_r}\, \sum_{\xB\sim\overline{y}\in X_r}\, \frac{|u(\yB)-u(\xB)|^2}{r^2}\mu_r(\{\xB\})\right)^{1/2}
  - \left(C_2\, \inf_{g_f}\, \int_Xg_f^2\, d\mu\right)^{1/2}\right)_+\\
  \ge &\left(\left(\frac{1}{C_1}\inf_g\, \int_Xg^2\, d\mu\right)^{1/2} - \left(C_2\inf_{g_f}\, \int_Xg_f^2\, d\mu\right)^{1/2}\right)_+ = \Psi(v)^{1/2},
\end{align*}
where we have used Theorem~\ref{thm:gbar-ug} in the last step above, and~\cite[Theorem~1.1(1)]{BLS} in the penultimate step.
\end{proof}
 
From the above lemma, with the choice of $F_k=\mathcal{E}_{r_k}$ for a choice of strictly monotone decreasing sequence of $r_k$ 
with $\lim_kr_k=0$,
the condition that $F_k\ge \Psi$ follows. Hence the above Lemma~\ref{lem:DalMaso} of Dal Maso~\cite{DalMaso}
tells us that there is a subsequence of $\mathcal{E}_{r_k}$ which $\Gamma$-converges to an energy $\mathcal{E}$ on $N^{1,2}_0(\Om)$.

\begin{remark}\label{rem:sequentialDalMaso}
Note that $N^{1,2}_0(\Om)$ is a Hilbert space, see for instance~\cite[Theorem~13.5.7]{HKSTbook}, or \cite[Corollary 3.27]{BrezisBook2011}, 
and is separable since Lipschitz functions are dense in that class~\cite[Theorem~8.2.1]{HKSTbook}. 
Therefore, by~\cite[Proposition~8.7]{DalMaso} we know the existence
of a metric $d$ on $N^{1,2}_0(\Om)$ for which the weak topology on $N^{1,2}_0(\Om)$, restricted to norm-bounded subsets, agrees
with the $d$-metric topology on that subset. It now follows from~\cite[Proposition~8.10]{DalMaso} that the $\Gamma$-convergence of
$\mathcal{E}_{r_k}$ to $\mathcal{E}$ satisfies properties~(e) and~(f) identified in~\cite[Proposition~8.1]{DalMaso}. 
That is,
\begin{enumerate}
\item[(i)] Whenever $(u_k)_k$ is a sequence in $N^{1,2}(X)$ and $u\in N^{1,2}_0(\Om)$ such that $u_k\to u$ weakly in $N^{1,2}_0(\Om)$, 
we must have $\mathcal{E}(u)\le \liminf_k \mathcal{E}_{r_k}(u_k)$;
\item[(ii)] For each $u\in N^{1,2}_0(\Om)$ there is a sequence $(u_k)_k$ in $N^{1,2}_0(\Om)$ such that $u_k\to u$ weakly in $N^{1,2}_0(\Om)$ and
$\mathcal{E}(u)\ge \limsup_k\mathcal{E}_{r_k}(u_k)$.
\end{enumerate}
\end{remark}

\section{Proof of Theorem~\ref{thm:main}: Convergence of graph energy minimizers.}\label{Sec:5}

Note that for each fixed $r>0$, the set $\Omega_r$ is finite, and therefore, the space $D^{1,2}_0(\Omega_r)$ is finite-dimensional. For any $\overline{v}\in D^{1,2}_0(\Omega_r)$, consider $E_r(\overline{v})$ 
\begin{equation*}
    \begin{aligned}
        E_r(\overline{v}) & = \sum_{\overline{x}\in X_r}\, \sum_{\overline{x}\sim \overline{y}\in X_r}\, \frac{|(\overline{v}+f_r)(\overline{y})-(\overline{v}+f_r)(\overline{x})|^2}{r^2}\, \mu_r(\{\overline{x}\}) \\ 
        & = \sum_{\overline{x}\in X_r}\, \sum_{\overline{x}\sim \overline{y}\in X_r}\, [\overline{v}^2(\xB)+ \overline{v}^2(\yB) - 2 \overline{v}(\xB)\cdot \overline{v}(\yB)] \frac{\mu_r(\{\overline{x}\}) }{r^2} \\ 
        & + \sum_{\overline{x}\in X_r}\, \sum_{\overline{x}\sim \overline{y}\in X_r}\,  [2 \overline{v}(\xB) f_r(\xB) +  2 \overline{v}(\yB) f_r(\yB) - 2 \overline{v}(\xB) f_r(\yB) -  2 \overline{v}(\yB) f_r(\xB)] \frac{\mu_r(\{\overline{x}\}) }{r^2} \\
        & + \sum_{\overline{x}\in X_r}\, \sum_{\overline{x}\sim \overline{y}\in X_r}\, [f_r^2(\xB)+ f_r^2(\yB) - 2 f_r(\xB)\cdot f_r(\yB)] \frac{\mu_r(\{\overline{x}\}) }{r^2} 
    \end{aligned}
\end{equation*}

This can be simplified using the symmetry $\overline{x}\sim \overline{y}$ iff $\overline{y}\sim \overline{x}$ to  
\begin{equation*}
    \begin{aligned}
        E_r(\overline{v}) &  = \sum_{\overline{x}\in \Om_r}\, \left(\sum_{\overline{x}\sim \overline{y}\in \Om_r}\,\frac{\mu_r(\{\xB,\yB\})}{r^2}\right)\,\vB(\xB)^2\, 
          -\, \sum_{\overline{x}\in \Om_r}\, \sum_{\overline{x}\sim \overline{y}\in \Om_r}\, \frac{\mu_r(\{\xB,\yB\})}{r^2}\vB(\yB)\vB(\xB)\\ 
        & + 2 \sum_{\overline{x}\in \Om_r}\, \sum_{\overline{x}\sim \overline{y}\in \Om_r}\,  [f_r(\xB)  - f_r(\yB)]\, \overline{v}(\xB)\,\frac{\mu_r(\{\overline{x},\yB\}) }{r^2} \\
        & + E_r(0).
    \end{aligned}
\end{equation*}
If we label the values $\overline{v}(\overline{x})$ by $v_i$, then the above shows that  
\begin{equation*}
    E_r(\overline{v}) =\sum_{i,j} a_{i,j} v_i v_j + \sum_{i} b_i v_i  +  E_r(0), 
\end{equation*}
where both sums are finite, $a_{i,j}$ are coefficients of a symmetric positive definite matrix and $b_i$ are some coefficients depending on $f$.

The minimizer $\uB[r]$ of this quadratic function can be found by solving $\nabla E_r(\uB[r])=0$ to get for each positive integer $i$,
\begin{equation*}
    \sum_{j} a_{i,j} \uB[r]_j + b_i  = 0,  
\end{equation*}  
and thus $\uB[r] = A^{-1} b$, where $A^{-1}$ is the inverse of the $a_{i,j}$ matrix.

The above discussion shows that for each $r>0$ we can find $\overline{u}[r]\in D^{1,2}_0(\Om_r)$ such that  we have
\[
E_r(\overline{u}[r])\le E_r(\overline{v}), \qquad \forall \overline{v}\in D^{1,2}_0(\Om_r).
\]

Projecting $\overline{u}[r]$ to $N^{1,2}_0(\Om)$ by $P_r$, we get 
\[
\mathcal{E}_r(P_r\overline{u}[r])=\widehat{\mathcal{E}_r}(P_r\overline{u}[r])\le \mathcal{E}_r(P_r \overline{v}), \qquad \forall \overline{v}\in D^{1,2}_0(\Om_r).
\]

Let us now fix a sequence $r_k \to 0^+$ as $k\to \infty$ and denote $u_k:=P_{r_k}\overline{u}[r_k]$. Then from the previous line,
\begin{equation}\label{eq:minimalityuk}
    \mathcal{E}_{r_k}(u_k) \leq \mathcal{E}_{r_k}(P_{r_k}\overline{v}), \qquad \forall \overline{v}\in D^{1,2}_0(\Om_r).
\end{equation}

\begin{lem}
    There exists $C>0$ such that for all positive integers $k$,
    \begin{equation*}
        \|u_k \|_{ N^{1,2}(\Omega)} < C.
    \end{equation*}
\end{lem}
\begin{proof}
    Applying the above line with $\overline{v} = 0$ and Lemma \ref{lemPsiE} we  get 
    \begin{equation*}
        \Psi(u_k) \leq \mathcal{E}_{r_k} u_{k} \leq \mathcal{E}_{r_k} 0.  
    \end{equation*}
    By~\cite[Theorem~1.1(1)]{BLS},
    there is a constant $C>0$ such that  
    \begin{equation*}
        \mathcal{E}_{r_k} 0  = E_{r_k} 0 \leq C \, \inf_{g_f}\, \int_X g^2_f d\mu,
    \end{equation*}
    with the infimum over all upper gradients $g_f$ of $f$.
    Hence, we get 
    \begin{equation*}
        \Psi(u_k) \leq C\, \inf_{g_f}\, \int_X g^2_f d\mu, 
    \end{equation*}
    and since by~\eqref{eq:DalMasoForce} the functional $\Psi$ is coercive, the result follows. 

\end{proof}

We follow the recipe listed below:
\begin{enumerate}
\item There is a subsequence, $(u_{k_j})_j$, 
that converges weakly in $N^{1,2}_0(\Om)$ to some function $u_\infty\in N^{1,2}_0(\Om)$
This is because of the reflexivity property of $N^{1,2}_0(\Om)$ and the
Banach-Alaoglu theorem,  see Remark~\ref{rem:sequentialDalMaso}. 
\item By Remark~\ref{rem:sequentialDalMaso}, we know that 
\[
\mathcal{E}(u_\infty)\le \liminf_{j\to\infty}\mathcal{E}_{r_{k_j}}(u_{k_j}).
\]
\end{enumerate}

\begin{lem}\label{lem:EnergyCVG}
We have that
\[
\mathcal{E}(u_\infty)= \lim_{j\to\infty}\mathcal{E}_{r_{k_j}}(u_{k_j})
\]
\end{lem}

\begin{proof}
We apply Remark~\ref{rem:sequentialDalMaso}(ii) to obtain
$v_j\in N^{1,2}_0(\Om)$ such that $(v_j)_j$ converges weakly to $u_\infty$ and so that
$\lim_{j\to\infty}\mathcal{E}_{r_{k_j}}(v_j)=\mathcal{E}(u_\infty)$. By the minimality property~\eqref{eq:minimalityuk} of $u_{k_j}$, we have that
\[
\mathcal{E}(u_\infty)=\lim_{j\to\infty}\mathcal{E}_{r_{k_j}}(v_j)\ge \lim_{j\to\infty}\mathcal{E}_{r_{k_j}}(u_{k_j})\ge \mathcal{E}(u_\infty),
\]
and so we have that 
\[
\mathcal{E}(u_\infty)=\lim_{j\to\infty}\mathcal{E}_{r_{k_j}}(u_{k_j}).
\]
\end{proof}

Recall that the energy $\mathcal{E}$ acts on $N^{1,2}_0(\Om)$ by adding in the contribution of $f$ to functions in that function class.

\begin{lem}
The function $u_\infty$ is an $\mathcal{E}$-energy minimizer on $\Om$. 
\end{lem}

\begin{proof}
For each $w\in N^{1,2}_0(\Om)$, we want to show that $\mathcal{E}(u_\infty)\le \mathcal{E}(w)$ to show that $u_\infty$ is
a $\mathcal{E}$-minimizer (and so $u_\infty+f$ is the $\mathcal{E}$-harmonic function with boundary values $f$).

Let $w\in N^{1,2}_0(\Om)$.  
If $\mathcal{E}(w)=\infty$, then 
the desired inequality follows, so without loss of generality we can assume that $\mathcal{E}(w)<\infty$. By~(ii) of 
Remark~\ref{rem:sequentialDalMaso}, we can find $w_j\in N^{1,2}_0(\Om)$ such that $(w_j)$ weakly converges to $w$ and
\[
\lim_j\mathcal{E}_{r_{k_j}}(w_j)=\mathcal{E}(w).
\]
So for sufficiently large $j$ we know that $\mathcal{E}_{r_{k_j}}(w_j)$ is finite, and so we can find
$\overline{w_j}\in D^{1,2}_0(\Om_{r_{k_j}})$ such that $w_j=P_{r_{k_j}}(\overline{w_j})$; note that then
by~\eqref{eq:minimalityuk},
\[
\mathcal{E}_{r_{k_j}}(w_j)=E_{r_{k_j}}(\overline{w_j})\ge E_{r_{k_j}}(\uB[r_{k_j}])=\mathcal{E}_{r_{k_j}}(u_{k_j}).
\]
Thus, by Lemma~\ref{lem:EnergyCVG}, we see that
\[
\mathcal{E}(u_\infty)=\liminf_{j\to\infty} \mathcal{E}_{r_{k_j}}(u_{k_j})\le \liminf_{j\to\infty}\mathcal{E}_{r_{k_j}}(w_j)=\mathcal{E}(w).\qedhere
\]
\end{proof}

The above lemmas together complete the proof of Theorem~\ref{thm:main}.

\end{document}